\documentclass[conference, letterpaper]{IEEEtran}
\IEEEoverridecommandlockouts
\usepackage{cite}
\usepackage[T1]{fontenc}
\usepackage{amsmath,amssymb,amsfonts, amsthm}
\usepackage{algorithmic}
\usepackage{graphicx}
\usepackage{textcomp}
\usepackage{xcolor}
\usepackage{dsfont}
\usepackage{booktabs}
\usepackage{hyperref}
\usepackage[caption=false,font=footnotesize]{subfig}
\def\BibTeX{{\rm B\kern-.05em{\sc i\kern-.025em b}\kern-.08em
    T\kern-.1667em\lower.7ex\hbox{E}\kern-.125emX}}

\DeclareMathOperator{\sgn}{sgn}

\newtheorem{definition}{Definition}
\newtheorem{prop}{Proposition}
\newtheorem{theorem}{Theorem}
\newtheorem{lemma}{Lemma}

\renewcommand{\IEEEtitletopspaceextra}{10pt}

\begin{document}

\title{\Large
A Framework for Exploring Social Interactions in Multiagent Decision-Making for Two-Queue Systems
}

\author{Mallory E. Gaspard, Naomi Ehrich Leonard% <-this % stops a space
\thanks{The work in this paper was supported by a gift from William H. Miller III. Research was sponsored by the Army Research Office and was accomplished under Grant Number W911NF2410126. The views and conclusions contained in this document are those of the authors and should not be interpreted as representing the official policies, either expressed or implied, of the Army Research Office or the U.S. Government. The U.S. Government is authorized to reproduce and distribute reprints for Government purposes notwithstanding any copyright notation herein.}
\thanks{M.~Gaspard is with the Dept. of Mechanical and
Aerospace Engineering and the Dept. of Ecology and Evolutionary Biology at Princeton University, Princeton, NJ, 08544 USA; \href{mailto:mallory.gaspard@princeton.edu}{mallory.gaspard@princeton.edu}.}%
\thanks{N.~E.~Leonard is with the Dept. of Mechanical and
Aerospace Engineering at Princeton University, Princeton, NJ, 08544 USA; \href{mailto:naomi@princeton.edu}{naomi@princeton.edu}. }%
}

\maketitle

\maketitle

\begin{abstract}
We introduce a new framework for multiagent decision-making in queueing systems that leverages the agility and robustness of nonlinear opinion dynamics to break indecision during queue selection and to capture the influence of social interactions on collective behavior. Queueing models are central to understanding multiagent behavior in service settings. Many prior models assume that each agent's decision-making process is optimization-based and governed by rational responses to changes in the queueing system. Instead, we introduce an internal opinion state, driven by nonlinear opinion dynamics, that represents the evolving strength of the agent's preference between two available queues. The opinion state is influenced by social interactions, which can modify purely rational responses. We propose a new subclass of queueing models in which each agent's behavioral decisions (e.g., joining or switching queues) are determined by this evolving opinion state. 
We prove a sufficient parameter condition that guarantees the Markov chain describing the evolving opinion and queueing system states reaches the Nash equilibrium of an underlying congestion game in finite expected time. We then explore the richness of the new framework through numerical simulations that illustrate the role of social interactions and an individual's access to system information in shaping collective behavior.
\end{abstract}

\section{Introduction}
From airport security screening lanes \cite{liu2025deep} to task allocation in distributed computing \cite{mitzenmacher2002power}, queueing systems are ubiquitous in everyday life. Mathematical models of these systems are crucial tools for understanding an individual's decision-making process in service settings, as well as for gaining insight into how these individual choices shape the collective behavior.
Decision-making in queueing systems typically involves agents choosing actions (e.g., joining a queue, switching queues, or leaving altogether)\cite{shortle2018fundamentals} that lead them to one of several \emph{mutually exclusive} physical states. 
Their choices are often motivated by a cost or reward associated with being in particular locations within the system. 
For example, an agent may decide to switch to a different queue if it is shorter or if another queue seems to have faster service \cite{koenigsberg1966jockeying}.    

Classical queueing frameworks typically assume that agents are \emph{rational} in deciding their actions.
Foundational models such as those presented in \cite{naor1969regulation, chr1972individual, edelson1975congestion, lin1984optimal} derive threshold-based decision policies from individual optimizations of each agent's expected reward and waiting cost. Other common approaches such as those surveyed in \cite{hassin2003queue} capture strategic decision-making in a non-cooperative game formulation and characterize desired behavior through Nash equilibria.
However, psychological studies have demonstrated that humans do not rationally respond to changes in a queueing system \cite{maister1984psychology, zhou2003looking}.
Feelings of stress, unfair treatment, and post-decision discomfort can lead to irrational decision-making by prompting unnecessary switching \cite{li2025people}, driving waiting time misperception \cite{larson1987or}, and supporting biased attachment to a particular option \cite{carmon2003option}.
 
Human service systems are driven by feedback between congestion and individual decision-making, which motivates behavioral queueing models based on bounded rationality and emotional influence \cite{simon1955behavioral, huang2013bounded, xue2023personality}. Although useful, these approaches often aggregate the environmental and social factors that impact the decision-making process, and thus have limited capacity to describe how each agent responds \emph{individually and dynamically} to these changing factors when choosing actions. To address this, we formulate a queueing framework in which each agent's decision-making process is governed by a continuously evolving \emph{internal state} modeled by a nonlinear dynamical system. We utilize the \emph{nonlinear opinion dynamics} (NOD) framework developed in \cite{bizyaeva2022nonlinear} whose dynamics are motivated by classical neural models in computational neuroscience (see e.g., \cite{hopfield1982neural, wilson1972excitatory}) to capture the agent's internal state as an \emph{opinion} reflecting their preference for one or the other queue in a two-queue system. Moreover, the NOD dynamical system has a bifurcation structure that supports multistability and hysteresis \cite{bizyaeva2022nonlinear} and provides a mechanism for \emph{neutrality breaking} in settings where an agent might find the available queues to be equally desirable choices. 

Unlike models which rely on rationality assumptions or optimization-based approaches, any rational actions that arise in our framework are \emph{emergent} phenomena that occur as a result of the underlying opinion dynamics. Recent work by Sinhmar et al. \cite{sinhmar2026neuromorphic} considers deterministic finite-action games and shows that under suitable conditions, best response occurs as the stable outcome of an internal state-space process governed by NOD. Here, we explore this notion further in the stochastic realm, where the NOD drives \emph{stochastic} joining and switching behaviors between queues, and we identify a parameter regime in which our framework can also act as a learning mechanism for Nash equilibrium queue configurations.

Our primary contribution is the introduction of an agent based queueing model (ABQM) in which the probability of an agent successfully executing a maneuver to one of two available queues is directly informed by their evolving opinion state (Section \ref{sec:nod_queues}). 
Given the range of possible collective behaviors that can arise from the NOD-driven maneuver dynamics, we discuss the interpretation of queue selection as a congestion game and highlight how Nash-equilibrium queue membership configurations can emerge from this dynamic process (Section \ref{sec:choice_and_games}). As a proof of concept and to maintain the focus on how the NOD shapes the movement behaviors, our analysis centers on the simplest setting with two queues, a fixed population, and no service \footnote{Ongoing work extends the analysis presented here to a more realistic $M$-queue setting with dynamic arrivals, nonzero service rates, and the opportunity for agents to willingly leave the system altogether without being served.}.  We cast the ABQM as a Markov chain and utilize standard martingale and probability arguments to prove a sufficient parameter condition that permits the system to reach a Nash equilibrium in finite expected time (Section \ref{sec:nash_hitting}). Although our condition is restrictive, we illustrate through numerical simulation the existence of a wider parameter regime that leads to the same results, and we highlight scenarios in which the ABQM behaves consistently with a stochastic learning process in congestion games (Section \ref{sec:numerical_sims}). We conclude and mention future directions in Section \ref{sec:conclusion}

\section{Opinion-Driven Selection Between Two Queues}\label{sec:nod_queues}
Consider a queueing system with two single-server queues ($A$ and $B$) and a waiting pool $W$ for agents who have not yet joined a queue. 
The system starts with $N$ agents, and no new agents are allowed to enter the system.
Once an agent joins a queue, they are allowed to switch between queues until they are served.
We assume that each server operates according to a Poisson point process with rate $\mu_q, q \in \{A, B\}$ that serves agents uniformly at random when a service event occurs.
Agents leave the system completely after service. 

\textit{Notation:} We denote vectors and matrices in \textbf{bold}. Agent $i$'s physical state at time $t$ is $\ell_i(t) \in \{0, \pm1\}$ 
where $0, +1,$ and $-1$ index $W$, $A$, and $B$, respectively.
Agent $i$'s internal state is represented by the opinion $z_i(t) \in [-1,1]$.
The physical and internal states of the $N$-agent group at time $t$ are encoded in the vectors $\boldsymbol{\ell}(t) := (\ell_1(t), \dots, \ell_N(t)) \subseteq \{0, \pm 1\}^N$ and $\boldsymbol{z}(t) := (z_1(t), \dots, z_N(t)) \subseteq [-1, 1]^N$, respectively. The social connections between each agent are represented by an $N$-node network described by the adjacency matrix $\boldsymbol{A} := [a_{ij}] \in \mathbb{R}^{N \times N}$. $a_{ij}$ reflects the strength of the influence that agent $j$ has on the internal state of agent $i$, and $\sgn(a_{ij})$ indicates whether this influence is attractive ($+$) or repulsive ($-$). For simplicity, we restrict $a_{ij} \in \{0, \pm1\}$. 
The matrix infinity norm is defined as $\|\boldsymbol{A}\|_{\infty} := \max_i \sum_{j}|a_{ij}|$.
Lastly, for $a, b \in \mathbb{R}$, the statement $a \wedge b$ is equivalent to $\min(a,b)$.
\subsection{Nonlinear Opinion Dynamics}
We utilize \emph{nonlinear opinion dynamics} (NOD) \cite{bizyaeva2022nonlinear}to model the continuous evolution of each agent's internal state.
Here, the opinion quantifies the agent's preference for one of two \emph{mutually exclusive} options ($A$ or $B$) and is influenced by self-reinforcement, social connections, and the changing environment.
Following the setup for the NOD system presented in \cite{bizyaeva2022nonlinear}, 
agent $i$'s opinion evolves according to the ordinary differential equation
\begin{equation}\label{eqn:nod_queue_ode}
   \dot{z}_i = -\lambda_iz_i + \tanh \left(\omega_i u_i  z_i + \alpha_i \sum^N_{j = 1, j\neq i}a_{ij}\hat{z}_j - \gamma_i b(\cdot) \right) 
\end{equation}
where $u_i = u_{0_i} + Kz^2_i$ represents the \emph{attention} that agent $i$ pays to its own internal state, $u_{0_i} \in \mathbb{R}$ is agent $i$'s basal attention, and $K 
\ge 0$ is the gain. $\hat{z}_j$ is agent $i$'s estimate of agent $j$'s opinion. In practice, $\hat{z}_j$ can be estimated by rapid analysis of sensory signals such as body orientation, gaze, and hand gestures \cite{langton2000eyes, csibra2007obsessed}.
The parameter $\lambda_i \in \mathbb{R}_{0,+}$ is a decay coefficient, $\omega_i \in \mathbb{R}_{0,+}$ is a self-reinforcement weight, $\gamma_i$ weighs the importance of environmental information, and $\alpha_i \in \mathbb{R}_{0,+}$ is the social influence weight. The function $b(\cdot)$ captures the  time-varying environmental input that agent $i$ receives from observations about the queueing system state (e.g., queue imbalance, wait time difference, or a combination of other metrics). 
In the ABQM, agents make decisions that affect queue states at discrete times, so Equation \eqref{eqn:nod_queue_ode} is formally a hybrid dynamical system in which the opinions evolve in continuous time, while the environmental input $b(\cdot)$ is piecewise constant between the decision epochs. 

When $b(\cdot) = 0$ and the basal attention is below a critical value, the neutral state ($\boldsymbol{z} = \boldsymbol{0}$) is a stable equilibrium of Equation \eqref{eqn:nod_queue_ode}. As shown in \cite{bizyaeva2022nonlinear}, this equilibrium can be \emph{destabilized} when the basal attention exceeds a critical threshold $\boldsymbol{u}^*_0$ or the input $b(\cdot)$ is large enough, i.e., greater than a $u_0$-dependent implicit threshold. 
Destabilization then induces opinion formation through a \emph{pitchfork} bifurcation, which also serves as a \emph{neutrality breaking} mechanism when agents regard all options as equally preferable. 
When the attention gain $K$ lies below a critical value $K^*$, the resulting bifurcation is a supercritical pitchfork with two stable branches corresponding to a strong preference for $A$ ($z_i > 0$) and a strong preference for $B$ ($z_i < 0$).
When the population is \emph{homogeneous}, closed-form expressions for $\boldsymbol{u}^*_0$ and $K^*$ are available. Here, $u^*_{0_i} = \frac{\lambda - \alpha \overline{\sigma}(\boldsymbol{A})}{\omega}$, where $\overline{\sigma}(\boldsymbol{A})$ is the largest eigenvalue of $\boldsymbol{A}$. A Lyapunov-Schmidt reduction further reveals that the bifurcation is a supercritical pitchfork when $K < K^* = \frac{\lambda^3}{3 \omega}$.  

\subsection{Agent Based Opinion-Driven Queueing Model}
We model the queueing process as an agent-based model over a finite time horizon $[0,T]$. 
While each agent's opinion evolves continuously according to Equation \eqref{eqn:nod_queue_ode}, 
agents only select actions (e.g., join or switch) at discrete decision epochs, $k \in \{0, \dots, M\}$. 
Decision epochs have equal spacing $\Delta t_D$, and the time \emph{between} epochs is a decision interval, $[t_{k}, t_{k+1})$. 

Let $z^k_i$ denote the opinion of agent $i$ at $t_k$. Similarly, let $\ell^k_i$ represent agent $i$'s queue membership at $t_k$, and $b_k$ be input from the environment at $t_k$. The global state of the system at epoch $k$ is $(\boldsymbol{z}_k, \boldsymbol{\ell}_k)$.
The opinion state ODEs are continuously integrated over the decision interval $[t_k, t^-_{k+1}]$ where $t^-_{k+1}$ denotes the instant before the next epoch. The updated opinion value at the end of the interval is $z^{k+1}_i$.

At $t^-_{k+1}$, each agent selects an action to initiate a transition to their preferred queue. The outcome is \emph{probabilistic}, and the probability that the transition is successful is informed by the \emph{preference margin} (i.e., by how much the agent prefers their most desired queue to the second-best option). Since the opinion is over two mutually-exclusive options, the preference margin is $|z^{k+1}_i|$. Let $l \in \{\pm 1\}$ be the index of agent $i$'s preferred queue at the end of the decision interval. Then,
\begin{equation}
\text{P(agent $i$ joins queue $l$ at $t^-_{k+1}$)} :=
    \begin{cases}
|z^{k+1}_i|, & \text{if } \ell^k_i \neq l\\
1, & \text{if } \ell^k_i = l.
\end{cases} 
\end{equation}
The queue membership $\boldsymbol{\ell}_{k+1}$ is determined by the action outcomes, and $b_{k+1}$ is updated using the resulting membership counts.
The process repeats through all decision intervals. 

Since each agent's movement probability depends only on the current state $(\boldsymbol{z}_k, \boldsymbol{\ell}_k)$, and we use a time-explicit numerical integration scheme on the system defined in \eqref{eqn:nod_queue_ode} to obtain $\boldsymbol{z}_{k+1}$, the joint state $\mathcal{X}_k := (\boldsymbol{z}_k, \boldsymbol{\ell}_k)$ evolves as a time-homogeneous Markov chain on the state space $\mathcal{S} := [-1, 1]^N \times \{0, \pm 1\}^N$. 
This representation allows us to use techniques from stochastic processes to analyze collective behavior. 
The model also has a meaningful game theoretic interpretation. 
In particular, when $b_k$ captures queue membership difference, and once $W$ is empty, $\boldsymbol{\ell}_k$ can be interpreted as a strategy profile in which each agent chooses a queue and incurs a cost based on the resulting queue membership counts. 
This naturally describes an underlying \emph{congestion game}, and we aim to characterize parameter regimes and properties of the social network where the NOD-driven, stochastic movement dynamics lead to emergent queue configurations corresponding to the Nash equilibria. 

\section{Queue Selection and Congestion Games}\label{sec:choice_and_games}
We examine the connection between the NOD-driven ABQM and congestion games in the two-queue setting. Congestion games capture scenarios where agents must \emph{strategically} choose from a set of available resources with usage-dependent costs \cite{roughgarden2005selfish}. In our setting, each queue is an available resource with membership-dependent cost.
Rosenthal showed in \cite{rosenthal1973class} that congestion games are potential games \cite{monderer1996potential} and thus admit at least one pure-strategy Nash equilibrium.

\begin{definition}[Queue Selection Congestion Game]\label{def:congestion_game_queue_choice}
Consider $N$ agents choosing between two queues, $A$ and $B$.
Let $n_A$ and $n_B$ be the number of agents in $A$ and $B$, respectively.  
Denote the resource costs for $A$ and $B$ as $c_A(n_A)$ and $c_B(n_B)$, and assume they are nondecreasing functions of their arguments.
Each agent selects a queue, and the resulting membership vector $\boldsymbol{\ell} \in \{-1, 1\}^N$ is a Nash equilibrium if no agent can unilaterally decrease their cost by switching queues.
\end{definition}     
When resource costs are the same across queues, it is straightforward to determine the Nash equilibrium analytically.
\begin{prop}\label{pro:nash_configs}
Consider a congestion game where $N$ agents are choosing between two queues with identical resource costs, $c_A = c_B = c(\cdot)$. A queue membership configuration $(n_A, n_B)$ is a Nash equilibrium if and only if $|n_A - n_B| \leq 1$. 
\end{prop}
\begin{proof}
    Suppose $|n_A - n_B| > 1$. Without loss of generality, assume $n_A > n_B + 1$, and consider an agent in $A$. If the agent in $A$ switches to $B$, then the queue membership counts become $n_A - 1$ and $n_B + 1$, and the agent incurs a cost of $c(n_B + 1)$. Since $n_A > n_B +1$ and the resource costs are nondecreasing, it follows that $c(n_A) \ge c(n_B +1)$, with a strict inequality if $c(\cdot)$ is increasing.
Because the agent unilaterally decreased their cost by switching, this configuration cannot be a Nash equilibrium. 

    Conversely, suppose $|n_A - n_B| \leq 1$, and consider an agent in $A$. If the agent switches to $B$, then the new membership counts become $n_A - 1$ and $n_B + 1$. Since $n_B + 1 \ge n_A$ and the costs are nondecreasing, it follows that $c(n_B + 1) \ge c(n_A)$. Thus, the agent cannot unilaterally decrease their cost by switching  to $B$. The argument is symmetric for agents in $B$ switching to $A$, so the configuration $(n_A, n_B)$ is a Nash equilibrium.
\end{proof}
Unlike standard congestion game models, our agents do not directly optimize the resource costs. 
When an agent in the more expensive queue is sensitive to environmental input and $b(\cdot)$ captures queue membership imbalances, the input bias may shift their internal state and increase their switching probability. 
As such, an equilibrium queue configuration may \emph{emerge} as a result of the NOD-driven decision-making dynamics.
In this sense, our ABQM can be interpreted as a stochastic learning process whose dynamics may drive the system toward the Nash equilibrium of the underlying congestion game. 
In the next section, we establish a sufficient condition involving the NOD parameters and the social network structure to provide finite expected Nash hitting time guarantees. 

\section{Sufficient Condition for Hitting the Set of Nash Configurations in Finite Time}\label{sec:nash_hitting}
Determining when $\{\mathcal{X}_k\}_{k \ge 0}$ reaches a Nash configuration in finite expected time establishes when the ABQM acts as a stochastic learning mechanism through which agents may naturally settle in optimal configurations, even when they are not explicitly optimizing anything.
Let $\Delta Q := n_A - n_B$ denote the signed queue imbalance. At a Nash equilibrium, the \emph{magnitude} of the imbalance is
\begin{equation*}
   \Delta Q_* := \begin{cases}
0, & \text{if } N \text{ is even}\\
1, & \text{if } N \text{ is odd}.
\end{cases} 
\end{equation*}
Outside of $\mathcal{N}$, the \emph{minimum} queue imbalance magnitude is $b_{\#} := \min_{\mathcal{X} \notin \mathcal{N}}\{b_k\} = 2$.
\begin{definition}[Nash Configuration Band]
    Consider the congestion game corresponding to selection between two identical queues. Let $n_A$, $n_B$, $n_W$ be the number of agents in $A$, $B$, and $W$ respectively. Denote the Markov chain state as $\mathcal{X}$. The band containing the Nash configurations for this game is
    \begin{equation}
        \mathcal{N} := \{\mathcal{X} \in \mathcal{S} \, : \, n_W = 0 \text{ and } |\Delta Q| \leq \Delta Q_*\}.
    \end{equation}
\end{definition}
We assume that every agent has complete information about the system. They can perfectly compute $\hat{z}_j$ for all other agents and they have exact knowledge of the signed queue imbalance $b_k = \Delta Q_k$ at each epoch. 
\begin{lemma}[Input Sign Alignment with the Cheaper Queue]\label{lemma:cheaper_queue_input_alignment}
    Let $\sigma_k := -\sgn(b_k)$. If 
    \begin{equation}
        \gamma_i b_{\#} \ge \omega_i(u_{0_i} + K) + \alpha_i\|\boldsymbol{A}\|_{\infty} + \overline{m}
    \end{equation}
    is satisfied for some $\overline{m} > 0$ and all agents $i$, then for any state $\mathcal{X}_k \notin \mathcal{N}$,
    \begin{equation*}
        \eta_i(t) := \sigma_k\left(\omega_i u_i z_i + \alpha_i \sum^N_{j = 1, j\neq i}a_{ij}\hat{z}_j - \gamma_i b_k \right) \ge \overline{m}
    \end{equation*}
    holds over $[t_k, t_{k+1})$. That is, the input to $\tanh(\cdot)$ in the opinion formation ODE is aligned with the cheaper queue.
\end{lemma}
\begin{proof}
    Since $|z_i| \leq 1$, $u_i \leq (u_{0_i} + K)$. Also, 
    $|\hat{z}_j| \leq 1$ yields $|\sum^N_{j = 1, j\neq i}a_{ij}\hat{z}_j| \leq \|\boldsymbol{A}\|_{\infty}$. By definition of $\sigma_k$, $\sigma_k(-\gamma_i b_k) = \gamma_i|b_k|$. So, $\eta_i(t)$ satisfies
    \begin{equation*}
        \eta_i(t) \ge \gamma_i|b_k| - \omega_i(u_{0_i} + K) - \alpha_i\|\boldsymbol{A}\|_{\infty}.
    \end{equation*}
    When $\mathcal{X}_k \notin \mathcal{N}$, $|b_k| \ge b_{\#}$, thus over $[t_k, t_{k+1})$,
    \begin{equation*}
        \eta_i(t) \ge \gamma_i b_{\#} - \omega_i(u_{0_i} + K) - \alpha_i\|\boldsymbol{A}\|_{\infty} \ge \overline{m}. \qedhere
    \end{equation*}
\end{proof}

\begin{lemma}[Uniform Opinion Magnitude Bound]\label{lemma:opinion_magnitude_bound}
    Assume the condition in Lemma \ref{lemma:cheaper_queue_input_alignment} holds for all $i$ and some $\overline{m}$. 
    Recall $\Delta t_D := t_{k+1} - t_k$ and define $        \beta_i := -e^{-\lambda_i\Delta t_D} + \frac{\tanh(\overline{m})}{\lambda_i}(1 - e^{-\lambda_i\Delta t_D})$.
    If $\Delta t_D$ satisfies
    \begin{equation}\label{eqn:timestep_restrict}
        \Delta t_D > \max_i\left\{\frac{1}{\lambda_i}\ln\left(\frac{\lambda_i + \tanh(\overline{m})}{\tanh(\overline{m})}\right)\right\}
    \end{equation}
    then, there exists a uniform constant $\beta > 0$ such that $\sigma_kz^{k+1}_i \ge \beta$ and thus $|z^{k+1}_i| \ge \beta$ for all $i$.
\end{lemma}
\begin{proof}
    To avoid treating the sign cases for $b_k$ separately, recall $\sigma_k := -\sgn(b_k)$ and define $y_i(t_k) = \sigma_kz^k_i$. Since $b_k$ is constant on $[t_k, t_{k+1})$, $\sigma_k$ is also constant over the interval.
    Let $x_i$ denote the input of $\tanh(\cdot)$ and $\tanh(\overline{m}) = \nu$. Since $\tanh$ is odd, $\sigma_k\tanh(x_i) = \tanh(\sigma_k x_i) = \tanh(\eta_i)$. By the condition in Lemma \ref{lemma:cheaper_queue_input_alignment}, we have that $\tanh(\eta_i(t)) \ge \nu$. So, $\dot{y}_i + \lambda_i y_i \ge \nu$. We multiply both sides of the inequality by the integrating factor $e^{\lambda_it}$. Integrating over the decision interval $[t_k, t_{k+1})$ and dividing both sides by $e^{\lambda_i t_{k+1}}$ yields
    \begin{equation}
        y_i(t^-_{k+1}) \ge e^{-\lambda_i\Delta t_D}y_i(t_k) + \frac{\nu}{\lambda_i}(1 - e^{-\lambda_i\Delta t_D}). 
    \end{equation}
    Using the bound $y_{i}(t_{k}) \in [-1,1] \; \forall k$, we obtain 
    \begin{equation}
        y_i(t^-_{k+1}) \ge  -e^{-\lambda_i(\Delta t_D)} + \frac{\nu}{\lambda_i}(1 - e^{-\lambda_i\Delta t_D}) = \beta_i.
    \end{equation} 
    Thus, $\sigma_kz^{k+1}_i \ge \beta_i$. 
    Define $\beta := \min_i\{\beta_i\}$. 
    Bounding $\beta_i$ away from zero for all agents yields the sufficient condition
    \begin{equation}
        \Delta t_D > \max_i \left\{\frac{1}{\lambda_i}\ln\left(\frac{\lambda_i + \nu}{\nu}\right)\right\}.
    \end{equation}
    When the condition above is satisfied, every agent has $\sigma_kz^{k+1}_i \ge \beta$ and
    \begin{equation}
        |z^{k+1}_i| \ge \beta, \hspace{0.1cm} \forall i.
    \end{equation}
    Thus, when the decision interval is sufficiently long, every agent's opinion is uniformly bounded away from neutrality in favor of the cheaper queue.
\end{proof}

\begin{theorem}[Nash Hitting in Finite Expected Time]\label{thm:finite_hit_time}
    Assume $\mathcal{X}_0 \notin \mathcal{N}$. Let $W_k$ denote the set of agents in $W$ at $k$. Suppose there exists a constant $\zeta > 0$ such that when $n^k_W > 0$, there is at least one agent $i \in W_k$ with $|z^{k+1}_i| \ge \zeta$.
    If there exists an $\overline{m} > 0$ such that 
    \begin{equation}\label{eqn:suff_input_condition}
        \gamma_i b_{\#} \ge \omega_i(u_{0_i} + K) + \alpha_i\|\boldsymbol{A}\|_{\infty} + \overline{m}
    \end{equation}
    holds for all $i$, the decision interval duration $\Delta t_D$ satisfies the condition in Lemma \ref{lemma:opinion_magnitude_bound}, and there exists a constant $\psi \in (0,1)$ such that $|z^{k+1}_i| \leq \psi$, for all agents in the more expensive queue when $\mathcal{X}_k \notin \mathcal{N}$ and $n^k_W = 0$, 
    then the Markov chain $\{\mathcal{X}_k\}_{k \ge 0}$ will hit $\mathcal{N}$ in finite expected time.
\end{theorem}
Intuitively, provided that $\Delta t_D$ is long enough for agents to respond internally before acting, condition \eqref{eqn:suff_input_condition} requires queue imbalance information to be significant enough to overcome self-reinforcement and social pressures in order to reach an equilibrium configuration. Using standard drift arguments \cite{meyn2012markov} and probability techniques \cite{williams1991probability}, we prove Theorem \ref{thm:finite_hit_time} by first bounding the expected number of epochs to empty $W$ followed by the expected number of additional epochs to hit $\mathcal{N}$. 
\begin{proof}
    Define the natural filtration $\mathcal{F}_k = \sigma(\mathcal{X}_0, \dots, \mathcal{X}_k)$. Consider the queueing system starting from an initial state $\mathcal{X}_0 \notin \mathcal{N}$ with $n^0_W \leq N$.
    Let $p^k_i = |z^{k+1}_i|$ denote the probability of agent $i \in W_k$ joining a queue during the decision interval $[t_{k}, t_{k+1})$, where $\sgn(z^{k+1}_i)$ indicates the destination. Define $\tau_W := \inf\{k \ge 0 \, | \, n^k_W = 0\}$ and
    $J^k_W$ as a sum of Bernoulli random variables denoting the number of agents that leave the waiting pool during the decision interval. 
    Since agents cannot join $W$, $n^{k+1}_W = n^k_W - J^k_W$ and the number of agents in $W$ is monotonically nonincreasing.
    By assumption, there is at least one $i \in W_k$ with $p^k_i \ge \zeta$, so
    \begin{equation*}
        \mathbb{E}[n^{k+1}_W - n^k_W \, | \, \mathcal{F}_k] = -\mathbb{E}[J^k_W \, | \, \mathcal{F}_k] = -\sum_{i \in W_k} p^k_i \leq - \zeta.
    \end{equation*}
    Define the stopped process $\mathcal{W}_k = n^{k \wedge \tau_W}_W + \zeta(k \wedge \tau_W)$. $\mathcal{W}_k$ is a supermartingale, and thus
    \begin{equation*}
        \mathbb{E}[n^{k \wedge \tau_W}_W + \zeta(k \wedge \tau_W)] \leq n^0_W.
    \end{equation*}
    Since $\mathbb{E}[n^{k \wedge \tau_W}_W] \ge 0$ and $n^0_W \leq N$, $\zeta \mathbb{E}[k \wedge \tau_W] \leq N.$
    $k \wedge \tau_W$ approaches $\tau_W$ almost surely as $k \rightarrow \infty$, so by the Monotone Convergence Theorem, we conclude
    \begin{equation}\label{eqn:expected_steps_to_clear_w}
        \mathbb{E}[\tau_{W} ] \leq N/\zeta. 
    \end{equation}
    Now, consider the process starting at $\tau_W$. Define $\tau_{\mathcal{N}} := \inf\{k \ge \tau_W \, | \, \mathcal{X}_k \in \mathcal{N}\}$. 
    Let $\widetilde{Q}_k$ denote the set of agents in the \emph{more expensive} queue at epoch $k$ and $|\widetilde{Q}_k|$ be its cardinality.
    Let $M_k$ be the number of switches from $\widetilde{Q}_k$ and $s_k$ be the number of switches from $\widetilde{Q}_k$ needed to hit $\mathcal{N}$ by the next epoch. 
    Given $\mathcal{F}_k$, $M_k$ is a sum of conditionally independent Bernoulli random variables, and by the conclusion of Lemma \ref{lemma:opinion_magnitude_bound}, and the assumption that $|z^{k+1}_i| \leq \psi$, $\forall i \in \widetilde{Q}_k$ when $\mathcal{X}_k \notin \mathcal{N}$ and $n^k_W = 0$, $\mathbb{P}(\mathcal{X}_{k+1} \in \mathcal{N} \, | \, \mathcal{F}_k) \ge  {|\widetilde{Q}_k| \choose s_k}\beta^{s_k}(1 - \psi)^{|\widetilde{Q}_k| - s_k}$.
    Thus, the probability that the Markov chain hits $\mathcal{N}$ by the \emph{next} epoch is uniformly bounded below as
    \begin{equation}
        p_{\mathcal{N}} := \min_{1 \leq q \leq N}\min_{1 \leq s \leq q}{q \choose s}\beta^{s}(1 - \psi)^{q - s}.
    \end{equation}
    Let $\Upsilon = \tau_{\mathcal{N}} - \tau_W$. So, for $\tau_W < k < \tau_{\mathcal{N}}$, $\mathbb{P}(\mathcal{X}_{k+1} \in \mathcal{N} \, | \, \mathcal{F}_k)  \ge p_{\mathcal{N}}$.
    Further, $\forall \, m \ge 0$, $\mathbb{P}(\Upsilon > m +1 \, | \, \mathcal{F}_{\tau_W + m}) \leq (1 - p_{\mathcal{N}})\mathds{1}_{\Upsilon > m}$. Taking the expectation of both sides yields 
    \begin{equation}\label{eqn:prob_bound_hit_n}
        \mathbb{P}(\Upsilon > m + 1) \leq (1- p_{\mathcal{N}})\mathbb{P}(\Upsilon > m).
    \end{equation}
    Induction on $m$ gives $\mathbb{P}(\Upsilon > m) \leq (1-p_{\mathcal{N}})^m$.
    By the tail-sum formula and geometric series, we have
    \begin{equation}\label{eqn:exp_additional_steps_to_N}
        \mathbb{E}[\Upsilon] = \sum^{\infty}_{m = 0} \mathbb{P}(\Upsilon > m) \leq \sum^{\infty}_{m = 0} (1 - p_{\mathcal{N}})^m = \frac{1}{p_{\mathcal{N}}}.
    \end{equation}
    Combining equations \eqref{eqn:expected_steps_to_clear_w} and \eqref{eqn:exp_additional_steps_to_N}, we conclude
    \begin{equation}
        \mathbb{E}[\tau_{\mathcal{N}}] \leq \frac{N}{\zeta} + \frac{1}{p_{\mathcal{N}}} < \infty. \qedhere
    \end{equation}
\end{proof}
The sufficient condition presented in the statement of Theorem \ref{thm:finite_hit_time} is restrictive, and there are many parameter regimes in which it may be violated but the chain still hits $\mathcal{N}$ in finite time. 
Further, we conjecture that anti-cooperative social network structures may even support \emph{persistence} in $\mathcal{N}$, i.e., provide more than  just finite-time hitting.
We investigate these notions in Section \ref{sec:numerical_sims} through numerical simulations.

\section{Numerical Experiments}\label{sec:numerical_sims}
In all experiments, we consider $N = 10$ identical agents choosing between identical queues $A$ and $B$. All agents are initially in $W$ (i.e., $\boldsymbol{\ell}_0 = \boldsymbol{0}$), and we set both server rates $\mu_A = \mu_B = 0$ to focus on the simplest queue-selection congestion game. 
We run the ABQM simulation over a finite time horizon $T = 30$ with decision interval $\Delta t_D = 0.1$.
We integrate the NOD ODEs using a Runge-Kutta-45 (RK-45) scheme with timestep $\Delta t = 0.01$.
For all agents, we fix $\gamma_i = 0.5$, $\omega_i = 1$, $\lambda_i = 1$, and $\alpha_i = 0.2$. 
This regime corresponds to agents whose opinions are shaped by joint influences from social connections, self-reinforcement, and queue imbalance information.
The critical parameters are $K^* = 1/3$ and
$u^*_{0_i} = 1$ for the anti-cooperative network. 
Taking $u_{0_i} = 1.25$ and $K = 0.25$ for all $i$ places the system slightly above the bifurcation point and ensures a supercritical pitchfork. 
For each agent, $z^0_i$ is sampled from a mean-zero normal distribution with standard deviation $0.1$. 

Our simulations aim to explore two main questions:
\begin{enumerate}
    \item Are cooperative or anti-cooperative network structures more conducive to stabilizing queue balance?
    \item How does limited access to queue information affect the system's ability to hit and persist in $\mathcal{N}$?
\end{enumerate}
We consider fully cooperative networks ($\boldsymbol{A}_+ := \boldsymbol{1}\boldsymbol{1}^T - \boldsymbol{I}$) and anti-cooperative networks ($\boldsymbol{A}_- := -\boldsymbol{1}\boldsymbol{1}^T + \boldsymbol{I}$). 
To model incomplete information, we set $b_k = 0$, $\forall k$ for a random selection of $N_b := \rho N$ agents, where $\rho \in \{0, 0.2, 0.4, 0.6, 0.8, 1\}$. 
For each network structure and each $\rho$, we run $10,000$ trials and record the average hitting time ($\overline{\tau}_{\mathcal{N}}$), its standard deviation ($\sigma(\overline{\tau}_{\mathcal{N}}))$, fraction of trials that hit $\mathcal{N}$ ($r$), average number of queue switches per agent ($\overline{S}$), and the average amount of time that the system stays at $\mathcal{N}$ after the last recorded hit ($\overline{T}_{\mathcal{N}}$). 
Tables \ref{tab:simulation_stats_coop} and \ref{tab:simulation_stats_anticoop} present these statistics (rounded to two decimal places).

\begin{table}[t]
\centering
\caption{Summary Statistics: Cooperative Network}
\label{tab:simulation_stats_coop}
\footnotesize
\begin{tabular}{l|ccccc}
\toprule
 $\rho$ & $\overline{\tau}_{\mathcal{N}}$ & $\sigma(\overline{\tau}_{\mathcal{N}})$ & $r$ & $\overline{S}$ & $\overline{T}_{\mathcal{N}}$ \\
\midrule
1: $\rho = 0$ & 2.96 & 1.12 & 1 & 21.96 & 0.06 \\
2: $\rho = 0.2$ & 3.01 & 1.17 & 1 & 18.01 & 0.12 \\
3: $\rho = 0.4$ & 2.95 & 1.13 & 1 & 12.56 & 0.22 \\
4: $\rho = 0.6$ & 3.45 & 1.19 & 0.74 & 1.58 & 0 \\
5: $\rho = 0.8$ & 2.8 & 1.05 & 0.12 & 0.47 & 0 \\
6: $\rho = 1$ & 1.958 & 0.50 & 0.003 & 0.12 & 0  \\
\bottomrule
\end{tabular}
\end{table}

\begin{table}[t]
\centering
\caption{Summary Statistics: Anti-Cooperative Network}
\label{tab:simulation_stats_anticoop}
\footnotesize
\begin{tabular}{l|ccccc}
\toprule
 $\rho$ & $\overline{\tau}_{\mathcal{N}}$ & $\sigma(\overline{\tau}_{\mathcal{N}})$ & $r$ & $\overline{S}$ & $\overline{T}_{\mathcal{N}}$ \\
\midrule
1: $\rho = 0$ & 3.01 & 1.25 & 1 & 4.51 & 16.92 \\
2: $\rho = 0.2$ & 3.07 & 1.21 & 1 & 4.53 & 17.03 \\
3: $\rho = 0.4$ & 3.17 & 1.29 & 1 & 3.88 & 17.68 \\
4: $\rho = 0.6$ & 3.43 & 1.63 & 0.9 & 2.05 & 20.46 \\
5: $\rho = 0.8$ & 5.74 & 4.00 & 0.76 & 0.51 & 22.50 \\
6: $\rho = 1$ & 7.32 & 4.53 & 0.86 & 0.04 & 22.56  \\
\bottomrule
\end{tabular}
\end{table}

\textit{Discussion:} Data in Tables \ref{tab:simulation_stats_coop} and \ref{tab:simulation_stats_anticoop}, indicate that the network structure itself does not have a significant influence on whether the agents are \emph{able} to hit a Nash configuration in finite time, but anti-cooperative network structures are more conducive to queue balance stabilization overall. Even when less than half of the population (e.g., $\rho \ge 0.6$) has access to imbalance information, the anti-cooperative agents still reach a Nash configuration at a high rate, while cooperative agents in limited information regimes often fail to hit $\mathcal{N}$ before $T$. 

Persistence in $\mathcal{N}$ is heavily influenced by both network structure and information access.
Cooperative agents do not exhibit any notable persistence in $\mathcal{N}$, while anti-cooperative agents always persist in $\mathcal{N}$ and the duration \emph{increases} as information access decreases.
Figures \ref{fig:anti_op_queue_rho_0pt6} and \ref{fig:full_op_queue_rho_0pt6} illustrate this behavior when $\rho = 0.6$.    
 \begin{figure}[h]
     \centering
     \includegraphics[width=0.95\columnwidth]{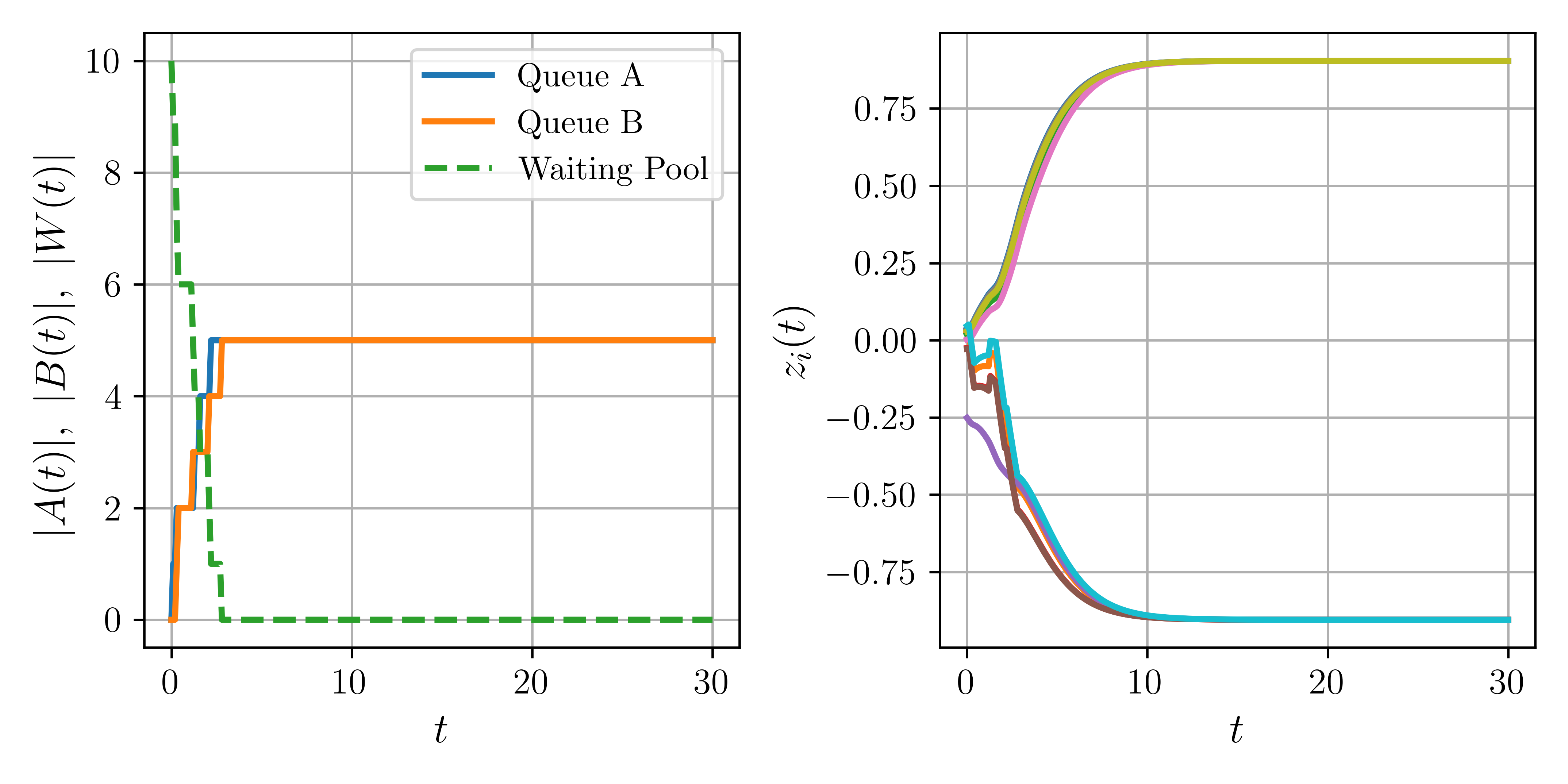}
     \vspace{-0.4cm}
     \caption{Queue lengths (left) and per-agent opinions (right) over time for $10$ anti-cooperative agents. $6$ agents lack queue imbalance information. Agents polarize quickly, do not switch queues, and easily settle in $\mathcal{N}$.}
     \label{fig:anti_op_queue_rho_0pt6}
 \end{figure}
 \begin{figure}[h]
     \centering    \includegraphics[width=0.95\columnwidth]{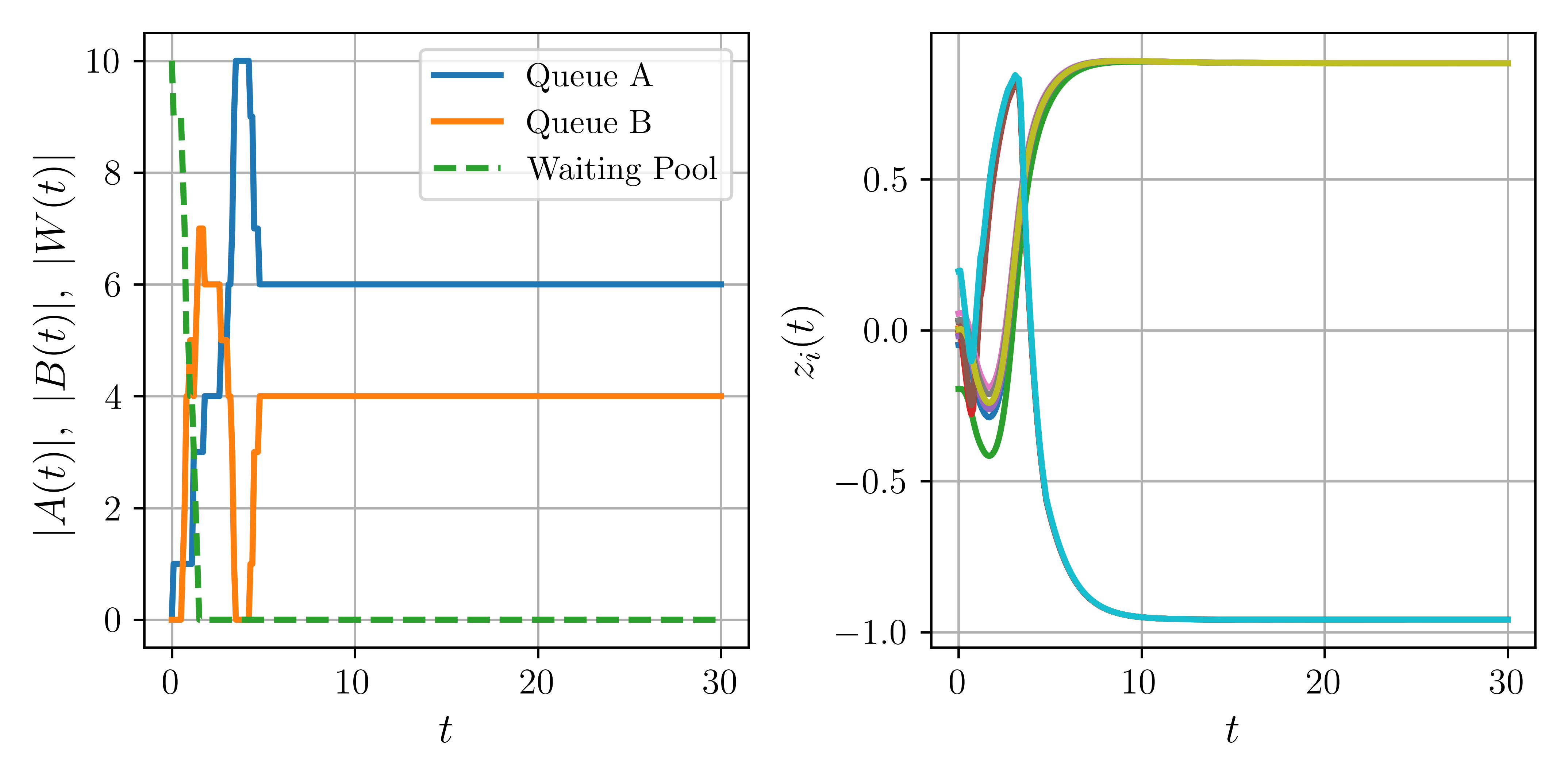}
     \vspace{-0.4cm}
     \caption{Queue lengths (left) and per-agent opinions (right) over time for $10$ fully-cooperative agents when $6$ agents lack imbalance information. After initial herding near $t = 4$ where all agents end up in $A$, the agents rapidly settle into two distinct groups based on their access to imbalance information.}
     \label{fig:full_op_queue_rho_0pt6}
 \end{figure}
Despite only four agents having information access, the anti-cooperative agents quickly polarize into equal queue lengths without any switching.
This illustrative example and the data in Table \ref{tab:simulation_stats_anticoop} suggest that \emph{anti-cooperative} agents can successfully \emph{learn} a Nash equilibrium of the underlying queue-selection congestion game in finite time.
On the other hand, a group of cooperative agents herd at around $t = 4$ when four agents simultaneously move from $B$ to $A$, resulting in \emph{all} agents ending up in the same queue. 
Then, the four \emph{informed} agents
quickly react by settling in $B$ to avoid the high imbalance cost in $A$.
This behavior reflects a fundamental tradeoff between responding rationally to environmental information and maintaining commitment to social connections,
particularly in cooperative networks. 
\section{Conclusion and Future Directions}\label{sec:conclusion}
We introduced an agent-based queueing framework where each agent's actions are driven by a continuously evolving internal opinion state. We proved a sufficient condition under which the system hits a Nash equilibrium configuration in finite expected time, and we numerically explored how social network structure and information access influence agent behavior. 
Our results suggest that network structure has limited impact on the group's ability to reach a Nash equilibrium in finite time, but anti-cooperative networks support persistence in a Nash configuration after the final hit. Further, as information access \emph{decreases}, anti-cooperative agents tend to settle into persistent polarization more quickly, while cooperative agents herd before splitting into groups reflecting their access to queue information. Future work extends the framework and Markov chain analysis to the generalized $M$-queue setting with non-identical queues, service, and dynamic arrivals while also examining mixed-sign social networks, relational structure between options, and empirical validation through data collection in real-world systems.

\section*{Acknowledgment}
The authors thank Vaibhav Srivastava for the helpful discussion about the supermartingale arguments in Section IV.

\bibliographystyle{IEEEtran}
\bibliography{bibl}

\end{document}